\documentclass[a4paper]{amsart}
\usepackage{amssymb, mathtools, amsthm}
\usepackage[hidelinks]{hyperref}
\usepackage{tikz}
\usepackage{graphicx}
\usepackage{subcaption}
\usetikzlibrary{decorations.markings}

\usepackage{tikz}
\usetikzlibrary{arrows.meta,calc}

\newcommand{\InnerRadius}{0.80}
\pgfmathsetmacro{\OuterRadius}{(2+sqrt(3))*\InnerRadius}

\newcommand{\CubeProjectionBase}{%
  \draw[gray!55,thin]
    (0:\InnerRadius) -- (90:\InnerRadius) --
    (180:\InnerRadius) -- (270:\InnerRadius) -- cycle;
  \draw[gray!55,thin]
    (0:\OuterRadius) -- (90:\OuterRadius) --
    (180:\OuterRadius) -- (270:\OuterRadius) -- cycle;
  \draw[gray!55,thin]
    (0:\InnerRadius) -- (0:\OuterRadius);
  \draw[gray!55,thin]
    (90:\InnerRadius) -- (90:\OuterRadius);
  \draw[gray!55,thin]
    (180:\InnerRadius) -- (180:\OuterRadius);
  \draw[gray!55,thin]
    (270:\InnerRadius) -- (270:\OuterRadius);
}

\newcommand{\VertexArrow}[4]{%
  \coordinate (p#4) at (#2:#1);
  \fill (p#4) circle[radius=1.35pt];
  \node[font=\small] at ($(p#4)+(#2:4.2mm)$) {$p_{#4}$};

  \ifnum#3=1
    \draw[very thick,-{Stealth[length=2.2mm,width=1.7mm]}]
    	($(p#4)+({#2-90}:0mm)$) --
      ($(p#4)+({#2+90}:6mm)$);
  \else
    \draw[very thick,-{Stealth[length=2.2mm,width=1.7mm]}]
    	($(p#4)+({#2+90}:0mm)$) --
      ($(p#4)+({#2-90}:6mm)$);
  \fi
}

\newcommand{\ModeDiagram}[8]{%
  \begin{tikzpicture}[scale=0.5]
    \CubeProjectionBase

    \VertexArrow{\InnerRadius}{0}  {#1}{1}
    \VertexArrow{\InnerRadius}{90} {#2}{2}
    \VertexArrow{\InnerRadius}{180}{#3}{3}
    \VertexArrow{\InnerRadius}{270}{#4}{4}

    \VertexArrow{\OuterRadius}{0}  {#5}{5}
    \VertexArrow{\OuterRadius}{90} {#6}{6}
    \VertexArrow{\OuterRadius}{180}{#7}{7}
    \VertexArrow{\OuterRadius}{270}{#8}{8}
  \end{tikzpicture}%
}

\newcommand{\CP}{\mathbf{CP}}
\newcommand{\cO}{\mathcal{O}}
\newcommand{\FS}{\text{FS}}
\newcommand{\uS}{\mathbb{S}}

\newcommand{\laplace}{\Delta}
\newcommand{\Hess}{\nabla^2}

\newcommand{\dist}{\operatorname{dist}}

\newcommand{\re}{\operatorname{Re}}

\newcommand{\ii}{{\rm i}}
\newcommand{\dd}{{\rm d}}

\theoremstyle{plain} %text of this environment is typesetted in italics
\newtheorem*{theorem}{Theorem}
\newtheorem*{lemma}{Lemma}
\newtheorem*{conjecture}{Conjecture}
\newtheorem*{problem}{Problem}
\theoremstyle{definition} %text of this environment is typesetted in roman letters

\theoremstyle{remark}
\newtheorem{remark}{Remark}

\begin{document}

\title[P and D locally minimize Gauss variance]
{The Primitive and Diamond surfaces\\
locally minimize\\
the variance of Gaussian curvature}

\author{Hao Chen}
\email{chenhao5@shanghaitech.edu.cn}
\address{ShanghaiTech University, Shanghai, China}

\keywords{minimal surfaces}
\subjclass[2020]{Primary 53A10}

\date{\today}

\begin{abstract}
	We prove that the Schwarz' Primitive (P) and Diamond (D) surfaces are local
	minimizers of the variance of Gaussian curvature among local deformations
	within the moduli space of triply periodic minimal surfaces of genus 3
	(TPMSg3s). Our approach interprets the branch values of the Gauss map as a
	configuration of eight points on the sphere and expresses the variance of
	Gaussian curvature as the product of two integrals involving exponentials of
	Green's functions. We then show that the Hessian of this functional is
	positive definite at the cubic configuration when restricted to antipodal
	deformations, thereby establishing the local minimality of the P and D
	surfaces.  Along the other deformation directions, the Hessian is positive
	definite except for a two-dimensional eigenspace with slightly negative
	eigenvalues, leaving promising hope that the Gyroid is also a local
	minimizer.
\end{abstract}

\maketitle

\emph{Triply Periodic Minimal Surfaces} (TPMSs) are minimal surfaces in flat
3-tori.  The genus of a TPMS is at least three.  A TPMS of genus three (TPMSg3)
has a Weierstrass parameterization defined on a branched double cover of the
sphere with eight branch points.  The branched covering map is provided by the
Gauss map, so the branch points correspond to eight \emph{flat points} on the
TPMS where the Gaussian curvature vanishes.  We refer the reader
to~\cite{meeks1990} for general reference.

TPMSs find applications in natural sciences as a model for periodic
bicontinuous structures~\cite{hyde1996}.  In nature and in laboratories, the
chance to observe Schwarz' Primitive (P), Diamond (D) and Schoen's Gyroid (G)
surfaces is much higher than finding other TPMSs.  Note that these surfaces
belong to the same Bonnet family.  They are therefore locally isometric, have
the same branch points, and the same variance of Gaussian curvature.  For all the
three surfaces, the branch points are vertices of a cube.  Based on numerical
evidence~\cite{schroderturk2006}, it was conjectured that

\begin{conjecture}
	Schwarz' Primitive, Diamond, and Schoen's Gyroid surfaces have the minimal
	variance of Gaussian curvature among all Triply Periodic Minimal Surfaces of
	genus 3.
\end{conjecture}

In this short note, we prove that the Primitive and Diamond are indeed local
minimizers.

\begin{theorem} \label{thm:main}
	Schwarz' Primitive and Diamond surfaces are local minimizers for the variance
	of Gaussian curvature among local deformations within the moduli space of Triply
	Periodic Minimal Surfaces of genus 3.
\end{theorem}

Our approach extends the variance of Gaussian curvature to an energy defined on
the $13$-dimensional space of configurations of eight points on the sphere up
to global rotations.  Only a five-dimensional locus of these configurations
corresponds to TPMSg3s.  More specifically, for a TPMSg3, the eight points
correspond to the normal vectors at the eight flat points.  In particular, the
five-dimensional subspace of antipodal configurations corresponds to embedded
TPMSg3s known as the \emph{Meeks family}.  For Diamond, Primitive, and Gyroid
surfaces, the eight points are at the vertices of a cube.

In~\cite{koiso2018}, it was proved that, up to homotheties and Euclidean
isometries, the Primitive, Diamond, and Gyroid surfaces all belong to a unique
locally rigid 5-parameter smooth family of pairwise non-homothetic triply
periodic minimal surfaces.  In particular, the family that contains the Diamond
and the Primitive surfaces is the Meeks family.  Schoen's Gyroid, however, does
not belong to Meeks family.

\medskip

At the cubic configuration, we compute the Hessian of our energy with respect
to local deformations of the configuration.  We find, unfortunately, that the
Hessian has a $2$-dimensional negative eigenspace generated by ``twists''.  So
the Hessian is indefinite, and the cubic configuration is not a local
minimizer.  Nevertheless, we are able to prove positive definiteness on the
subspace of antipodal deformations.  So the cubic configuration is a local
minimizer among antipodal deformations, thereby proving the Theorem.

We are unable to conclude on the Gyroid because its $5$-parameter family of
local deformations is not fully understood.  Among the known deformations of
the Gyroid, it is indeed a local minimizer.  But in the presence of the
negative eigenspace, we cannot rule out the possibility that it might not be a
local minimizer along other deformations.

\begin{remark}
	Our energy generalizes to the configurations of $n$ points on the sphere.
	See Section~\ref{sec:npoints} for details.  This is an interesting problem
	for discrete geometry in its own right.
\end{remark}

The paper is organized as follows.  In Section~\ref{sec:reformulate}, we will
reformulate the problem in terms of point configurations on the sphere.  The
Theorem is proved in Section~\ref{sec:proof}, but many complicated
computational details are postponed to the appendix.  We compute the gradient
of our energy in Appendix~\ref{app:gradient}, the Hessian in
Appendix~\ref{app:hessian}, and analyze the definiteness of the Hessian on each
of irreducible representations in Appendix~\ref{app:irreps}.

\subsection*{Acknowledgement}

The author thanks Chengjian Yao and Thierry De Pauw for very helpful and
inspiring discussions.  ChatGPT 5.6 Sol was used to assist in discovering the
algebraic symmetrization in Appendix C.4 and in producing the TikZ code for
Figure~\ref{fig:generators}. All identities, estimates, and proofs were
subsequently checked independently by the author.

\section{Reformulation}
\label{sec:reformulate}

\subsection{Weierstrass parameterization}

Given a TPMSg3, let $p_1, \cdots, p_8 \in \mathbb{C}$ be (the stereographic
projection of) the Gauss map at the eight flat points.  Consider the
hyperelliptic Riemann surface $\Sigma$ of genus three defined by
\[
	w^2 = P(z) = \prod_{i=1}^8 (z-p_i).
\]
It can be seen as a branched double cover of the Riemann sphere $\uS^2 \simeq
\CP^1$ with eight branch points.  Then we have the following \emph{Weierstrass
parameterization} for the TPMS (up to translations)
\begin{equation}\label{eq:weierstrass}
	\Sigma \ni (z, w) \mapsto \re \int^{(z, w)} \frac{(1-z^2, \ii(1+z^2), 2z)}{w}
	e^{i \vartheta}\, \dd z.
\end{equation}
Switching the two branches by $(z,w) \mapsto (z,-w)$ corresponds to the
inversion with respect to any of the eight flat points.  The parameter
$\vartheta$ is the \emph{Bonnet angle}; changing $\vartheta$ gives a locally
isometric deformation of the surface known as \emph{Bonnet rotation}.  Two
surfaces whose $\vartheta$ differ by $\pi/2$ are called \emph{conjugate pairs}.

For the minimal surface to be triply periodic and globally well-defined, one
needs to solve the period problem (aka, \emph{close the periods}).  More
precisely, one finds $p_i$'s and $\vartheta$ so that the real parts of the
integrals over closed cycles on $\Sigma$ form a three dimensional lattice.  The
period problem involves three real equations for each of the three coordinates,
hence nine real equations in total.  Since we have $17$ real parameters
($p_i$'s and $\vartheta$), the set of TPMSg3s forms a $8$-dimensional variety up
to translations and homotheties, or $5$-dimensional variety up to Euclidean isometries and homotheties.

Closing the periods is usually very difficult.  But Meeks proved
in~\cite{meeks1990} that, when the branch points form four antipodal pairs,
there exists a conjugate pair of embedded TPMSg3s.  Up to homotheties and
Euclidean isometries, these TPMSg3s form a $5$-dimensional manifold termed
\emph{Meeks family}.

In this paper, we will not insist on closing the periods.  The variance of
Gaussian curvature is invariant under local isometries such as Bonnet rotations.
So it is a function defined on the $13$-dimensional space of configurations of
eight points on the sphere up to rotations.  The Meeks surfaces correspond to
the five-dimensional locus of antipodal configurations, for which we know the
existence of $\vartheta$ that closes the periods.  Our main Theorem is
equivalent to that the cubic configuration is a local minimizer of the Gauss
variance on this $5$-dimensional subspace.

\subsection{Variance of Gaussian curvature}

More precisely, we want to minimize the (dimensionless) variance of Gauss
curvature, defined by
\[
	\sigma(K)
	= \frac{\langle K^2 \rangle - \langle K \rangle^2}{\langle K \rangle^2}
	= \Bigg[\frac{\int_\Sigma K^2 \, \dd A}{\int_\Sigma \dd A} - \Bigg( \frac{\int_\Sigma K \, \dd A}{\int_\Sigma \dd A} \Bigg)^2\Bigg] / \Bigg( \frac{\int_\Sigma K \, \dd A}{\int_\Sigma \dd A} \Bigg)^2.
\]
By Gauss-Bonnet, $\int_\Sigma K \, \dd A$ is a topological invariant.  It is then
equivalent to minimize
\begin{equation}\label{eq:form1}
	\int_\Sigma K^2 \dd A \times \int_\Sigma \dd A.
\end{equation}

From the Weierstrass parameterization, we compute
\[
	\dd A = \frac{(1+|z|^2)^2}{|w|^2} \, dx \wedge dy, \qquad K = -4 \frac{|w|^2}{(1+|z|^2)^4},
\]
where $z = x + y \ii$.  Indeed, one verifies that
\[
	\int_\Sigma K \dd A = -4 \int_\Sigma \frac{dx \wedge dy}{(1+|z|^2)^2} = - 8 \pi
\]
as expected.  The area element and the Gaussian curvature are invariant after
switching the branches.  It then suffices to compute the invariance on one of
the branches.  So \eqref{eq:form1} becomes (up to a constant factor)
\begin{equation}\label{eq:form2}
	\int_{\CP^1} \frac{|P(z)|}{(1+|z|^2)^6} \, dx \wedge dy \times
	\int_{\CP^1} \frac{(1+|z|^2)^2}{|P(z)|} \, dx \wedge dy.
\end{equation}
Note that the second integral is improper: The integrand is not defined at the
zeros of $P(z)$.  So the integral is defined as the limit
\[
	\int_{\CP^1} \frac{(1+|z|^2)^2}{|P(z)|} \, dx \wedge dy
	= \lim_{\varepsilon \to 0}
	\int_{\CP^1 \setminus \cup_i D^i_\varepsilon} \frac{(1+|z|^2)^2}{|P(z)|} \, dx \wedge dy,
\]
where $D^i_\varepsilon$ denote the $\varepsilon$-disks around $p_i$ under the
Fubini--Study metric.  Fortunately, the limit converges.

\subsection{Fubini--Study metric}

Note that
\[
	\dd A_\FS = \frac{4 \, dx \wedge dy}{(1+|z|^2)^2}
\]
is the area element for the Fubini--Study metric on $\CP^1$, which is
(normalized to) exactly the standard metric on the unit sphere $\uS^2$.
Moreover, we can identify $P(z)$ to a global section $s$ of the holomorphic
line bundle $\cO(8) = \cO(1)^{\otimes 8}$ over $\CP^1$, and
\[
	\|s\|_\FS = \frac{|P(z)|}{(1+|z|^2)^4}
\]
is its point-wise Fubini--Study norm.  See, for instance,
\cite{huybrechts2005}.  Then \eqref{eq:form2} becomes (up to a constant factor)
\begin{equation}\label{eq:form3}
	J_8 := 
	\int_{\CP^1} \| s \|_\FS \, \dd A_\FS \times
	\int_{\CP^1} \| s \|_\FS^{-1} \, \dd A_\FS. 
\end{equation}
Note again that the second integral is improper.  It is defined as the
convergent limit of the integral over $\CP^1 \setminus \cup_i D^i_\varepsilon$
as $\varepsilon \to 0$.

\subsection{Lelong--Poincar\'e Equation}

Define
\[
	u = \log \| s \|_\FS.
\]
Note that the Fubini--Study metric on $\cO(8)$ is Hermitian.  One then easily verifies the
Lelong--Poincar\'e Equation
\[
	\laplace u = 2\pi \sum_{i=1}^8 \delta(z-p_i) - 4\frac{4}{(1+|z|^2)^2},
\]
or, using the Laplacian under the Fubini--Study metric,
\[
	\laplace_\FS u = \frac{(1+|z|^2)^2}{4} \laplace u
	= 2\pi \sum_{i=1}^8 \delta_i(z) - 4,
\]
where
\[
	\delta_i(z) = \frac{(1+|z|^2)^2}{4} \delta(z-p_i)
\]
denote the Dirac delta function under the Fubini--Study metric centered at
$p_i$.

The solution is given by
\[
	u = 2\pi \sum_{i = 1}^8 G_i(z),
\]
where $G_i(z)$ is the Green function solving
\begin{equation}\label{eq:laplaceG}
	\laplace_\FS G_i(z) = \delta_i(z) - \frac{1}{4\pi}.
\end{equation}
On the unit sphere $\uS^2 \simeq \CP^1$, we have
\begin{equation}\label{eq:Green}
	G_i(z) = \frac{1}{2\pi} \log \sin \frac{\dist(z, p_i)}{2} + C
\end{equation}
where $\dist(z, p_i)$ is the spherical distance between $z$ and $p_i$.  Then
\eqref{eq:form3} becomes
\begin{equation}\label{eq:form4}
	J_8 =
	\int_{\uS^2} \exp(u) \, \dd A \times
	\int_{\uS^2} \exp(-u) \, \dd A,
\end{equation}
where $\dd A$ is the standard spherical area element on the unit sphere.
Note that \eqref{eq:form4} does not depend on the constant $C$ in $G_i$, so we
assume $C=0$ from now on.  Note again that the second integral is improper and
is defined as the convergent limit
\[
	\int_{\uS^2} \exp(-u) \, \dd A = \lim_{\varepsilon \to 0}
	\int_{\Omega_\varepsilon} \exp(-u) \, \dd A,
\]
where
\[
	\Omega_\varepsilon = \uS^2 \setminus \cup_{i=1}^n D^i_\varepsilon.
\]

The Conjecture would be proved if $J_8$ is minimized among all configurations
when $p_i \in \uS^2$ are the vertices of a cube.  This is, unfortunately, not
the case, as we will see later.

\begin{remark}
	With this reformulation, we actually extend the variance of Gaussian
	curvature beyond embedded TPMSg3s.  We are now considering all minimal
	surfaces with the Weierstrass parameterization~\eqref{eq:weierstrass}.  In
	particular, we do not need the period problems to be solved, not to mention
	the embeddedness.  The variance can be seen as an energy function defined on
	the space of point configurations.
\end{remark}

\subsection{Generalization}
\label{sec:npoints}

The analysis above can be generalized as follows.  Fix $n$ distinct points
$p_1, \cdots, p_n$, and define the polynomial
\[
	P(z) = \prod_{i=1}^n (z-p_i).
\]
Consider the product of integrals
\[
	\int_{\CP^1} \frac{|P(z)|}{(1+|z|^2)^{\frac{n}{2}+2}} \, dx \wedge dy \times
	\int_{\CP^1} \frac{(1+|z|^2)^{\frac{n}{2}-2}}{|P(z)|} \, dx \wedge dy.
\]
Note that
\[
	\|s\|_\FS = \frac{|P(z)|}{(1+|z|^2)^{n/2}}
\]
is the point-wise Fubini--Study norm of $P(z)$ as a global section $s$ of
$\cO(n)$ over $\CP^1$. Then the product can be rewritten (up to a constant
factor) into the form
\[
	J_n :=
	\int_{\CP^1} \| s \|_\FS \, \dd A_\FS \times
	\int_{\CP^1} \| s \|_\FS^{-1} \, \dd A_\FS. 
\]
Now the Lelong--Poincar\'e Equation is
\[
	\laplace_\FS u = 2\pi \sum_{i=1}^n \delta_i(z) - n/2,
\]
which is solved by
\[
	u = 2\pi \sum_{i=1}^n G_i(z).
\]
So
\[
	J_n = 
	\int_{\uS^2} \exp(u) \, \dd A \times
	\int_{\uS^2} \exp(-u) \, \dd A.
\]
And we could ask the following problem
\begin{problem}
	What are the positions of $p_1, \cdots, p_n$ that minimize $J_n$?
\end{problem}
This is an interesting problem in its own right.  When $n = 8$, this is exactly
the problem of finding branch values to minimize the variance of the Gauss
curvature.

\section{Proof}
\label{sec:proof}

We now present our proof to the Theorem.  Computational details are left to the
appendix.

\begin{proof}
	Recall that
	\[
		\Omega_\varepsilon = \uS^2 \setminus \cup_{i=1}^n D^i_\varepsilon,
	\]
	where $D^i_\varepsilon$ denote the $\varepsilon$-disks around $p_i$ under the
	spherical metric.  Define
	\[
		I^\pm_\varepsilon = \int_{\Omega_\varepsilon} \exp(\pm u) \, \dd A,
	\]
	and two probability measure $\dd\mu^\pm_\varepsilon = \rho^\pm_\varepsilon \dd
	A$ on $\Omega_\varepsilon$, where
	\[
		\rho_\varepsilon^\pm
		= \frac{\exp(\pm u)}{\int_{\Omega_\varepsilon} \exp(\pm u) \, \dd A}
		= \frac{\exp(\pm u)}{I^\pm_\varepsilon}.
	\]
	We need to compute the gradient and the Hessian of
	\[
		\log J_8 = \log I^+ + \log I_-.
	\]
	In Appendix~\ref{app:gradient} we compute the gradient of $\log I^\pm$ as
	follows:
	\begin{align}
		\langle \nabla_{\! p} \log I^\pm, v\rangle :=&
		\lim_{\varepsilon\to 0}\langle \nabla_{\! p} \log I^\pm_\varepsilon, v\rangle \nonumber\\
		=& \int_{\uS^2} \sum_{i=1}^n \langle \mp 2\pi \nabla_{\! z} G_i, \tilde v_i \rangle \dd\mu^\pm,\label{eq:gradient}
	\end{align}
	where
	\[
		v = (v_1, \cdots, v_8) \in \bigoplus_{i=1}^8 T_{p_i} \uS^2,
	\]
	and $\tilde v_i$ is a Killing field such that $\tilde v_i (p_i) = v_i$.  Note
	that the formula is independent of the choice of the Killing field.  To see
	this, assume another choice $\tilde v'_i$.  Then $\tilde w_i = \tilde v'_i -
	\tilde v_i$ is a Killing field that vanishes at $p_i$, which is given as a
	rotation around $p_i$.  Consequently, we have $\langle 2\pi \nabla_{\!  z}
	G_i, \tilde w_i \rangle = 0$.  Moreover, if $\tilde v_i = \tilde v_*$ is the
	same Killing field for all $i$, one easily verifies that the gradient
	vanishes.  This is expected because $J$ is invariant under global rotations.

	At the cubic configuration, the gradient may be identified with a tuple of
	eight tangent vectors, one at each $p_i$. The $O_h$-invariance of $I^\pm$
	implies that this tuple is invariant under the induced action of $O_h$,
	including its permutation of the vertices. The tangent representation has no
	nonzero invariant vector, so the gradient must vanish at the cube.
	Therefore, the cube configuration is a critical point of $\log I^\pm$ and of
	$\log J_8$.

	Then in Appendix~\ref{app:hessian}, we compute the Hessian of $\log I^\pm$ at
	a critical point as follows:
	\begin{align}
		\Hess_{\! p} \log I^\pm (v, v)
		:=& \lim_{\varepsilon \to 0} \Hess_{\! p} \log I^\pm_\varepsilon (v, v) \nonumber\\
		=& \int_{\uS^2} \sum_{i=1}^n \sum_{j=1}^n \langle 2\pi \nabla_{\! z} G_j, \tilde v_j - \tilde v_i \rangle \langle 2\pi \nabla_{\! z} G_i, \tilde v_i \rangle \, \dd \mu^\pm.\label{eq:hessian}
	\end{align}
	Again, the formula is independent of the choice of the Killing field by the
	same argument as above.  Moreover, if $\tilde v_i = \tilde v_*$ is the same
	Killing field for all $i$, one easily verifies that the Hessian vanishes.

	Without loss of generality, we assume that the vertices of the cube, in the
	standard coordinates $(\theta, \phi)$ of the sphere, are as follows:
	\begin{align*}
		p_1 &= (\theta_0, 0),&
		p_5 &= (\pi - \theta_0, 0),\\
		p_2 &= (\theta_0, \pi/2),&
		p_6 &= (\pi - \theta_0, \pi/2),\\
		p_3 &= (\theta_0, \pi),&
		p_7 &= (\pi - \theta_0, \pi),\\
		p_4 &= (\theta_0, 3\pi/2),&
		p_8 &= (\pi - \theta_0, 3\pi/2).
	\end{align*}
	where $\theta_0 = \arccos(1/\sqrt{3})$.

	We study the Hessian on the $13$-dimensional space
	\[
		V = \Big(\bigoplus_{i=1}^8 T_{p_i} \uS^2\Big) / \mathcal{R}
	\]
	at the cubic configuration, where
	\[
		\mathcal{R} = \{(R(p_1), \cdots, R(p_8)) \mid R \in \mathfrak{so}(3)\}
	\]
	is the subspace of deformations generated by global rotations.  For this
	purpose, we decompose $V$ into irreducible representations (irreps) of the
	symmetry group $O_h$ of the cube~\cite{atkins1970} as follows:

	\begin{itemize}
	% \item The $3$-dimensional space $T_{1g}$ generated by rotations.  One
	% 	generator is given by
	% 	\[
	% 		v_i = \frac{\partial_\phi}{\sin\theta}\Big|_{p_i} \quad 1 \le i \le 8.
	% 	\]
		\item The $3$-dimensional space $T_{1u}$ generated by slides (that moves the
			center of mass).  One generator is given by
			\[
			% v_i = \partial_\theta|_{p_i}, \quad 1 \le i \le 8.
				v_i = \frac{\partial_\phi}{\sin\theta}\Big|_{p_i} \text{ for } i \in \{1, 2, 5, 6\},\qquad
				v_i =-\frac{\partial_\phi}{\sin\theta}\Big|_{p_i} \text{ for } i \in \{3, 4, 7, 8\}.
			\]
		\item The $2$-dimensional space $E_g$ generated by the orbit of
			\[
				v_i = \frac{\partial_\phi}{\sin\theta}\Big|_{p_i} \text{ for } i \in \{1, 3, 5, 7\},\qquad
				v_i =-\frac{\partial_\phi}{\sin\theta}\Big|_{p_i} \text{ for } i \in \{2, 4, 6, 8\}.
			\]
			Every face remains rectangular under these deformations.
		\item The $2$-dimensional space $E_u$ generated by the orbit of
			\[
				v_i = \frac{\partial_\phi}{\sin\theta}\Big|_{p_i} \text{ for } i \in \{1, 2, 3, 4\},\qquad
				v_i =-\frac{\partial_\phi}{\sin\theta}\Big|_{p_i} \text{ for } i \in \{5, 6, 7, 8\}.
			\]
			The face centers of the cube remain symmetry centers under these deformations.
		\item The $3$-dimensional space $T_{2g}$ generated by the orbit of
			\[
			% v_i = \partial_\theta|_{p_i}, & \text{ for } i \in \{1, 3, 6, 8\},\qquad
			% v_i =-\partial_\theta|_{p_i}, & \text{ for } i \in \{2, 4, 5, 7\}.
				v_i = \frac{\partial_\phi}{\sin\theta}\Big|_{p_i} \text{ for } i \in \{1, 2, 7, 8\},\qquad
				v_i =-\frac{\partial_\phi}{\sin\theta}\Big|_{p_i} \text{ for } i \in \{3, 4, 5, 6\}.
			\]
		\item The $3$-dimensional space $T_{2u}$ generated by the orbit of
			\[
				v_i = \frac{\partial_\phi}{\sin\theta}\Big|_{p_i} \text{ for } i \in \{1, 3, 6, 8\},\qquad
				v_i =-\frac{\partial_\phi}{\sin\theta}\Big|_{p_i} \text{ for } i \in \{2, 4, 5, 7\}.
			\]
	\end{itemize}

	See Figure~\ref{fig:generators} for diagrams for the generators after
	stereographic projection.

	\begin{figure}[ht]
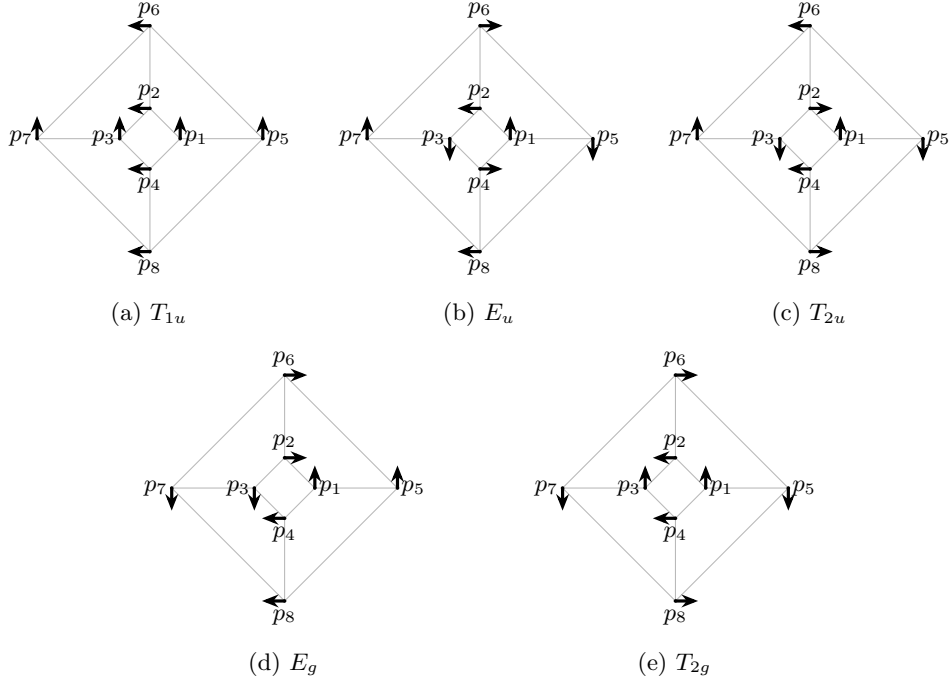

  	\centering

  % ---------- First row: u-modes ----------
  	\begin{subfigure}[t]{0.31\textwidth}
    	\centering
      \ModeDiagram
      { 1}{ 1}{-1}{-1}
      { 1}{ 1}{-1}{-1}%
    	\caption{$T_{1u}$}
    	\label{fig:T1u-mode}
  	\end{subfigure}\hfill
  	\begin{subfigure}[t]{0.31\textwidth}
    	\centering
      \ModeDiagram
      { 1}{ 1}{ 1}{ 1}
      {-1}{-1}{-1}{-1}%
    	\caption{$E_u$}
    	\label{fig:Eu-mode}
  	\end{subfigure}\hfill
  	\begin{subfigure}[t]{0.31\textwidth}
    	\centering
      \ModeDiagram
      { 1}{-1}{ 1}{-1}
      {-1}{ 1}{-1}{ 1}%
    	\caption{$T_{2u}$}
    	\label{fig:T2u-mode}
  	\end{subfigure}

  	\medskip

  % ---------- Second row: g-modes ----------
  	\begin{subfigure}[t]{0.31\textwidth}
    	\centering
      \ModeDiagram
      { 1}{-1}{ 1}{-1}
      { 1}{-1}{ 1}{-1}%
    	\caption{$E_g$}
    	\label{fig:Eg-mode}
  	\end{subfigure}
  	\hspace{0.08\textwidth}
  	\begin{subfigure}[t]{0.31\textwidth}
    	\centering
      \ModeDiagram
      { 1}{ 1}{-1}{-1}
      {-1}{-1}{ 1}{ 1}%
    	\caption{$T_{2g}$}
    	\label{fig:T2g-mode}
  	\end{subfigure}

  	\caption{Generators of the five deformation modes.}
  	\label{fig:generators}
	\end{figure}

	The Hessian is an $O_h$-invariant symmetric bilinear form. Since the five
	irreducible summands are pairwise non-isomorphic and occur with multiplicity
	one, by Schur's lemma, the restriction of the Hessian to each summand is a
	scalar multiple of identity.  So we only need to study the Hessian on one
	generator for each of the irreps.  The analysis of definiteness is carried
	out in Appendix~\ref{app:irreps}.  The analysis for $T_{2g}$
	(Appendix~\ref{app:T2g}) is assisted by ChatGPT 5.6 Sol.  The result is that
	the Hessian of $\log I^\pm$ at the cube is positive definite on all irreps
	except for $E_u$, where the Hessian is negative.

	By the local rigidity theorem of Koiso–Piccione–Shoda, every sufficiently
	small TPMSg3 deformation of P or D, up to homotheties and Euclidean
	isometries, belongs to a 5-parameter family, which can only be the Meeks
	family. So the corresponding point configuration remains antipodal. The
	tangent space to the locus of antipodal configurations at the cube is $E_g
	\oplus T_{2g}$.  Since the Hessian of $\log J_8$ is positive definite on this
	space, the cubic configuration is a strict local minimizer among antipodal
	configurations.

	This finishes the proof that P and D are local minimizers of the Gauss
	variance.
\end{proof}

For the G surface, its local deformations are not fully understood.  We now
know a tetragonal and a rhombohedral deformation.  The corresponding $v$ lies
in $E_u \oplus E_g$ and $E_u \oplus T_{2g}$, respectively.  Numerics show that
G is indeed a local minimizer for the Gauss variance along these
deformations~\cite{schroderturk2006}, meaning that the positive contributions of
$E_g$ and $T_{2g}$ outweigh the negative contribution of $E_u$.

Indeed, numerics show that the negative eigenvalue of $E_u$ is very close to
$0$.  The local deformations of $G$ correspond to a five-dimensional subspace
of the 13-dimensional $V$.  It is very promising to hope that the restriction
of the Hessian to the local deformations of $G$ is positive definite, so G is
indeed a local minimizer of the Gauss variance.  To confirm this, we need to
work towards a better understanding of the deformations of G.

\bibliography{References}
\bibliographystyle{plain}

\appendix

\section{The gradient}
\label{app:gradient}

Recall that
\[
	\Omega_\varepsilon = \uS^2 \setminus \cup_{i=1}^n D^i_\varepsilon,
\]
where $D^i_\varepsilon$ denote the $\varepsilon$-disks around $p_i$ under the
spherical metric.  Define
\[
	I^\pm_\varepsilon = \int_{\Omega_\varepsilon} \exp(\pm u) \, \dd A,
\]
and two probability measure $\dd\mu^\pm_\varepsilon = \rho^\pm_\varepsilon \dd
A$ on $\Omega_\varepsilon$, where
\[
	\rho_\varepsilon^\pm
	= \frac{\exp(\pm u)}{\int_{\Omega_\varepsilon} \exp(\pm u) \, \dd A}
	= \frac{\exp(\pm u)}{I^\pm_\varepsilon}.
\]

We compute
\begin{align*}
	\langle \nabla_{\! p_i} \log I^\pm_\varepsilon, v_i\rangle &=
	\frac{\langle \nabla_{\! p_i} I^\pm_\varepsilon, v_i\rangle}{I^\pm_\varepsilon}
	\\ &=
	\frac{1}{I^\pm_\varepsilon}
	\Big(
		\int_{\Omega_\varepsilon} \langle \pm 2\pi \nabla_{\! p_i} G_i, v_i \rangle \exp(\pm u) \, \dd A
		+ \oint_{\partial D^i_\varepsilon} \exp(\pm u) \langle n_i, \tilde v_i \rangle \dd s
	\Big)
	\\ &=
	\int_{\Omega_\varepsilon} \langle \pm 2\pi \nabla_{\! p_i} G_i, v_i \rangle \rho_\varepsilon^\pm \, \dd A
	+ \oint_{\partial D^i_\varepsilon} \langle n_i, \tilde v_i \rangle \rho_\varepsilon^\pm \dd s
\end{align*}
where $v_i \in T_{p_i} \uS^2$, $n_i$ is the unit normal vectors on the boundary
of $D^i_\varepsilon$ pointing to the interiors of the disks, and $\tilde v_i$
is a Killing field such that $\tilde v_i (p_i) = v_i$. So for
\[
	v = (v_1, \cdots, v_8) \in \bigoplus_{i=1}^8 T_{p_i} \uS^2,
\]
we have
\begin{equation}\label{eq:grad}
	\langle \nabla_{\! p} \log I^\pm_\varepsilon, v\rangle 
	= \int_{\Omega_\varepsilon} \sum_{i=1}^n \langle \mp 2\pi \nabla_{\! z} G_i, \tilde v_i \rangle \dd\mu_\varepsilon^\pm
	+ \sum_{i=1}^n \oint_{\partial D^i_\varepsilon} \langle n_i, \tilde v_i \rangle \rho_\varepsilon^\pm \dd s
\end{equation}

Note that the formula does not depend on the choice of the Killing fields
$\tilde v_i$.  To see this, assume another choice $\tilde v'_i$.  Then $\tilde
w_i = \tilde v'_i - \tilde v_i$ is a Killing field that vanishes at $p_i$,
which is given as a rotation around $p_i$.  Consequently, we have $\langle \mp
2\pi \nabla_{\! z} G_i, \tilde w_i \rangle = 0$ and $\langle n_i, \tilde w_i
\rangle = 0$.

Note also that, when $\tilde v_i = \tilde v_*$ is the same Killing field for
all $i$, then
\begin{align*}
	\langle \nabla_{\! p} \log I^\pm_\varepsilon, v\rangle 
	&= \int_{\Omega_\varepsilon} \sum_{i=1}^n \langle \mp 2\pi \nabla_{\! z} G_i, \tilde v_* \rangle \dd\mu_\varepsilon^\pm
	+ \sum_{i=1}^n \oint_{\partial D^i_\varepsilon} \langle n_i, \tilde v_* \rangle \rho_\varepsilon^\pm \dd s\\
	&= \int_{\Omega_\varepsilon} \langle \mp \nabla_{\! z} u, \tilde v_* \rangle \dd\mu_\varepsilon^\pm
	+ \int_{\Omega_\varepsilon} \langle \pm \nabla_{\! z} u, \tilde v_* \rangle \dd \mu_\varepsilon^\pm = 0
\end{align*}
So a deformation along a Killing field does not change $I^\pm_\varepsilon$.

\medskip

Recall that $2\pi G_i = \log \sin (\frac{1}{2} \dist(z, p_i))$, so
\[
	2\pi \nabla_{\! z} G_i \sim \frac{\hat w_i}{\dist(z, p_i)}, \qquad
	\rho^\pm_\varepsilon \sim \dist(z, p_i)^{\pm 1} \qquad
\]
as $z \to p_i$, where $\hat w_i = \nabla_z \dist(z, p_i)$.  It is then obvious
that~\eqref{eq:grad} converges as $\varepsilon \to 0$ for $\mu^+_\varepsilon$.
As for $\mu^-_\varepsilon$, in the standard spherical coordinate with the pole
$p_i$, polar angle $\theta$ and azimuth angle $\phi$, the integrands
\begin{align*}
	\langle 2\pi \nabla_{\! z} G_i, \tilde v_i \rangle \dd\mu_\varepsilon^-
	&\sim \frac{2}{\theta I^-} \langle \hat w_i(\theta, \phi), \tilde v_i(\theta, \phi) \rangle \dd \theta \dd \phi,\\
	\langle n_i, \tilde v_i \rangle \rho_\varepsilon^- \dd s
	&\sim \frac{-2}{I^-} \langle \hat w_i(\varepsilon, \phi), \tilde v_i(\varepsilon, \phi) \rangle \dd \phi
\end{align*}
as $\theta, \varepsilon \to 0$.  The integrals in~\eqref{eq:grad} then converge
as $\varepsilon \to 0$ because
\[
	\lim_{\theta \to 0} \int_0^{2\pi} \langle \hat w_i(\theta, \phi), \tilde v_i(\theta, \phi) \rangle \dd \phi \to 0.
\]
More specifically, we have
\[
	\langle \nabla_{\! p} \log I^\pm, v\rangle :=
	\lim_{\varepsilon\to 0}\langle \nabla_{\! p} \log I^\pm_\varepsilon, v\rangle 
	= \int_{\uS^2} \sum_{i=1}^n \langle \mp 2\pi \nabla_{\! z} G_i, \tilde v_i \rangle \dd\mu^\pm.
\]
As a sanity check, one easily verifies that this vanishes when $\tilde v_i =
\tilde v_*$ is the same Killing field for all $i$.

\medskip

Finally, by symmetry, one easily verifies that $\langle \nabla_{\! p} \log
I^\pm_\varepsilon, v\rangle = 0$ for any $v$ when $p_1,\cdots,p_8$ are the
vertices of the cube. So the cube is indeed a critical point of $J_8$.

\section{The Hessian}
\label{app:hessian}

We compute,
\begin{align*}
	\langle \nabla_{\! p_i} \rho^\pm_\varepsilon, v_i\rangle
	&=
	\frac{\langle \pm 2\pi \nabla_{\! p_i} G_i, v_i \rangle \exp(\pm u)}{I_\varepsilon^\pm}
	-
	\frac{\exp(\pm u) \langle \nabla_{\! p_i} I_\varepsilon^\pm, v_i\rangle}{(I_\varepsilon^\pm)^2}
	\\
	&= \rho_\varepsilon^\pm \Big(
		\langle \mp 2\pi \nabla_{\! z} G_i, \tilde v_i \rangle
		- \int_{\Omega_\varepsilon} \langle \mp 2\pi \nabla_{\! z} G_i, \tilde v_i \rangle \dd\mu_\varepsilon^\pm
		- \oint_{\partial D^i_\varepsilon} \langle n_i, \tilde v_i \rangle \rho_\varepsilon^\pm \dd s
	\Big)
\end{align*}

The Hessian
\begin{align}
	\Hess_{\! p} \log I^\pm_\varepsilon (v , v)
	=&
	\langle \nabla_{\! p} \langle \nabla_{\! p} \log I^\pm_\varepsilon, v\rangle, v \rangle \nonumber \\
	=&\phantom{+}
	\int_{\Omega_\varepsilon} \sum_{i=1}^n \langle \nabla_{\! p_i} \langle \pm 2\pi \nabla_{\! p_i} G_i, v_i \rangle, v_i \rangle \rho_\varepsilon^\pm \, \dd A \label{eq:l1}\\
	&+ \int_{\Omega_\varepsilon} \sum_{i=1}^n \langle \pm 2\pi \nabla_{\! p_i} G_i, v_i \rangle \sum_{j=1}^n \langle \nabla_{\! p_j} \rho_\varepsilon^\pm, v_j \rangle \, \dd A \label{eq:l2}\\
	&+ \sum_{i=1}^n \oint_{\partial D^i_\varepsilon} \langle n_i, \tilde v_i \rangle \sum_{j=1}^n \langle \nabla_{\! p_j} \rho_\varepsilon^\pm, v_j \rangle \dd s \label{eq:l3}\\
	&+ \sum_{j=1}^n \oint_{\partial D^j_\varepsilon} \langle n_j, \tilde v_j \rangle \sum_{i=1}^n \langle \pm 2\pi \nabla_{\! p_i} G_i, v_i \rangle \rho_\varepsilon^\pm \, \dd s \label{eq:l4}\\
	&+ \sum_{i=1}^n \oint_{\partial D^i_\varepsilon} \langle n_i, \tilde v_i \rangle \langle \nabla_{\! z} \rho_\varepsilon^\pm, \tilde v_i \rangle \dd s \nonumber
\end{align}

We expand each term
\begin{flalign*}
	\text{\eqref{eq:l1}}
	=& \int_{\Omega_\varepsilon} \sum_{i=1}^n \pm 2\pi \Hess_{\! z} G_i (\tilde v_i, \tilde v_i) \dd\mu_\varepsilon^\pm &&
\end{flalign*}
\begin{flalign*}
	\text{\eqref{eq:l2}}
	=& \int_{\Omega_\varepsilon} \sum_{i=1}^n \langle \mp 2\pi \nabla_{\! z} G_i, \tilde v_i \rangle \sum_{j=1}^n \langle \nabla_{\! p_j} \rho_\varepsilon^\pm, v_j \rangle \, \dd A &&\\
	=& \int_{\Omega_\varepsilon} \sum_{i=1}^n \langle \mp 2\pi \nabla_{\! z} G_i, \tilde v_i \rangle \sum_{j=1}^n \langle \mp 2\pi \nabla_{\! z} G_j, \tilde v_j \rangle \dd\mu_\varepsilon^\pm &&\\
	&- \int_{\Omega_\varepsilon} \sum_{i=1}^n \langle \mp 2\pi \nabla_{\! z} G_i, \tilde v_i \rangle \dd\mu_\varepsilon^\pm
	\int_{\Omega_\varepsilon} \sum_{j=1}^n \langle \mp 2\pi \nabla_{\! z} G_j, \tilde v_j \rangle \dd\mu_\varepsilon^\pm &&\\
	&- \int_{\Omega_\varepsilon} \sum_{i=1}^n \langle \mp 2\pi \nabla_{\! z} G_i, \tilde v_i \rangle \dd\mu_\varepsilon^\pm 
	\sum_{j=1}^n \oint_{\partial D^j_\varepsilon} \langle n_j, \tilde v_j \rangle \rho_\varepsilon^\pm \, \dd s &&
\end{flalign*}
\begin{flalign*}
	\text{\eqref{eq:l3}}
	=& \sum_{i=1}^n \oint_{\partial D^i_\varepsilon} \langle n_i, \tilde v_i \rangle \sum_{j=1}^n \langle \mp 2\pi \nabla_{\! z} G_j, \tilde v_j \rangle \rho_\varepsilon^\pm \, \dd s &&\\
	&- \sum_{i=1}^n \oint_{\partial D^i_\varepsilon} \langle n_i, \tilde v_i \rangle\rho_\varepsilon^\pm \, \dd s
	\int_{\Omega_\varepsilon} \sum_{j=1}^n \langle \mp 2\pi \nabla_{\! z} G_j, \tilde v_j \rangle \dd\mu_\varepsilon^\pm &&\\
	&- \sum_{i=1}^n \oint_{\partial D^i_\varepsilon} \langle n_i, \tilde v_i \rangle\rho_\varepsilon^\pm \, \dd s
	\sum_{j=1}^n \oint_{\partial D^j_\varepsilon} \langle n_j, \tilde v_j \rangle \rho_\varepsilon^\pm \, \dd s, &&
\end{flalign*}
and
\begin{flalign*}
	\text{\eqref{eq:l4}}
	=& \sum_{j=1}^n \int_{\partial D^j_\varepsilon} \langle n_j, \tilde v_j \rangle \sum_{i=1}^n \langle \mp 2\pi \nabla_{\! z} G_i, \tilde v_i \rangle \rho_\varepsilon^\pm \, \dd s. &&
\end{flalign*}

Putting them together gives
\begin{align}
	\Hess_{\! p} \log I^\pm_\varepsilon (v, v)
	=& \int_{\Omega_\varepsilon} \Big( \sum_{i=1}^n \langle 2\pi \nabla_{\! z} G_i, \tilde v_i \rangle \Big)^2 \dd\mu_\varepsilon^\pm \label{eq:h1}\\
	&+ \int_{\Omega_\varepsilon} \sum_{i=1}^n \pm 2\pi \Hess_{\! z} G_i (\tilde v_i, \tilde v_i) \dd\mu_\varepsilon^\pm  \label{eq:h2}\\
	&- \Big(
		\int_{\Omega_\varepsilon} \sum_{i=1}^n \langle \mp 2\pi \nabla_{\! z} G_i, \tilde v_i \rangle \dd\mu_\varepsilon^\pm
		+ \sum_{i=1}^n \oint_{\partial D^i_\varepsilon} \langle n_i, \tilde v_i \rangle \rho_\varepsilon^\pm \dd s
	\Big)^2 \label{eq:h3}\\
	&+ \sum_{i=1}^n \oint_{\partial D^i_\varepsilon} \langle n_i, \tilde v_i \rangle \sum_{j=1}^n \langle \mp 2\pi \nabla_{\! z} G_j, \tilde v_j \rangle \rho_\varepsilon^\pm \, \dd s  \label{eq:h4}\\
	&+ \sum_{i=1}^n \oint_{\partial D^i_\varepsilon} \langle n_i, \tilde v_i \rangle \sum_{j=1}^n \langle \mp 2\pi \nabla_{\! z} G_j, \tilde v_j - \tilde v_i \rangle \rho_\varepsilon^\pm \, \dd s \label{eq:h5}
\end{align}

Note again the independence of the choice of the Killing fields $\tilde v_i$
(by the same argument as before) and, when $\tilde v_i = \tilde v_*$ is the
same Killing field for all $i$,
\begin{multline*}
	\Hess_{\! p} \log I^\pm_\varepsilon (v, v)
	= \int_{\Omega_\varepsilon} \langle \nabla_{\! z} u, \tilde v_* \rangle^2 \dd\mu_\varepsilon^\pm
	+ \int_{\Omega_\varepsilon} \pm \Hess_{\! z} u (\tilde v_*, \tilde v_*) \dd\mu_\varepsilon^\pm\\
	+ \sum_{i=1}^n \oint_{\partial D^i_\varepsilon} \langle n_i, \tilde v_* \rangle \langle \mp \nabla_{\! z} u, \tilde v_* \rangle \rho_\varepsilon^\pm \, \dd s = 0.
\end{multline*}

The boundary term~\eqref{eq:h5} converges to $0$ because
$D^i_\varepsilon$ only contains the singularity of $G_j$ when $j=i$.  For the
same reason, in the limit $\varepsilon \to 0$, we may rewrite the ``coupling''
boundary term \eqref{eq:h4} into the form
\begin{align*}
	&\lim_{\varepsilon \to 0} \sum_{i=1}^n \oint_{\partial D^i_\varepsilon} \langle n_i, \tilde v_i \rangle \langle \mp 2\pi \nabla_{\! z} G_i, \tilde v_i \rangle \rho_\varepsilon^\pm \, \dd s\\
	=&\lim_{\varepsilon \to 0} \sum_{i=1}^n \Big[
		- \int_{\Omega_\varepsilon} \langle \nabla_{\! z} u, \tilde v_i \rangle \langle 2\pi \nabla_{\! z} G_i, \tilde v_i \rangle \, \dd \mu_\varepsilon^\pm
		\mp \int_{\Omega_\varepsilon} 2\pi \Hess_{\! z} G_i (\tilde v_i, \tilde v_i) \, \dd \mu_\varepsilon^\pm
	\Big],
\end{align*}
where the second term cancels with~\eqref{eq:h2}.

We have seen that the second term in~\eqref{eq:h3} converges to $0$ as
$\varepsilon \to 0$.  At a critical point, its first term~\eqref{eq:h3}
vanishes.  If this is the case, we have
\begin{align*}
	&\lim_{\varepsilon \to 0} \Hess_{\! p} \log I^\pm_\varepsilon (v, v) \\
	=&\lim_{\varepsilon \to 0}
	\int_{\Omega_\varepsilon} \Big( \sum_{i=1}^n \langle 2\pi \nabla_{\! z} G_i, \tilde v_i \rangle \Big)^2 \dd\mu_\varepsilon^\pm
	- \int_{\Omega_\varepsilon} \sum_{i=1}^n \langle \nabla_{\! z} u, \tilde v_i \rangle \langle 2\pi \nabla_{\! z} G_i, \tilde v_i \rangle \, \dd \mu_\varepsilon^\pm\\
	=&\lim_{\varepsilon \to 0}
	\int_{\Omega_\varepsilon} \sum_{i=1}^n \sum_{j=1}^n \langle 2\pi \nabla_{\! z} G_j, \tilde v_j - \tilde v_i \rangle \langle 2\pi \nabla_{\! z} G_i, \tilde v_i \rangle \, \dd \mu_\varepsilon^\pm.
\end{align*}
When $i \ne j$, the convergence follows from the same argument for the
convergence of the gradient~\eqref{eq:grad}.  The possible divergence with $i
= j$ is ruled out by $\tilde v_j - \tilde v_i$.  In the following, we will
simply write
\[
	\Hess_{\! p} \log I^\pm (v, v)
	=
	\int_{\uS^2} \sum_{i=1}^n \sum_{j=1}^n \langle 2\pi \nabla_{\! z} G_j, \tilde v_j - \tilde v_i \rangle \langle 2\pi \nabla_{\! z} G_i, \tilde v_i \rangle \, \dd \mu^\pm.
\]

\section{Definiteness of the Hessian on the irreps}
\label{app:irreps}

Recall from the proof the assumed positions of the vertices and the generators
of each irrep.

\subsection{The definiteness on \texorpdfstring{$E_u$}{Eu} and \texorpdfstring{$E_g$}{Eg}}

We have chosen the generator vectors, intentionally, with the following
property.  The eight vertices are partitioned into two subsets $A$ and $B$,
with four vertices each, such that vertices in $A$ extend to the Killing field
$\partial_\phi$, while vertices in $B$ extend to $-\partial_\phi$.  Then we
have
\begin{align*}
	&\Hess_{\! p} \log I^\pm (v, v)
	=
	\int_{\uS^2} \sum_{i=1}^n \sum_{j=1}^n \langle 2\pi \nabla_{\! z} G_j, \tilde v_j - \tilde v_i \rangle \langle 2\pi \nabla_{\! z} G_i, \tilde v_i \rangle \, \dd \mu^\pm\\
	=& -4
	\int_{\uS^2} \sum_{i \in A} \sum_{j \in B} (2\pi \partial_\phi G_j)(2\pi \partial_\phi G_i) \, \dd \mu^\pm
	= -16\pi^2 \int_{\uS^2} \partial_\phi G_A \partial_\phi G_B \, \dd \mu^\pm,
\end{align*}
where
\[
	G_A = \sum_{i \in A} G_i, \qquad G_B = \sum_{i \in B} G_i.
\]

One then easily observes that $\partial_\phi G_A \partial_\phi G_B$ is
everywhere non-positive for $E_g$, and everywhere non-negative for $E_u$.
Since the integrand is not identically zero, this proves that the Hessian is
positive definite on $E_g$, and negative definite on $E_u$.

\subsection{The definiteness on \texorpdfstring{$T_{1u}$}{T1u}}

To study the definiteness of on $T_{1u}$, we need to first study the sign of
\[
	H^\pm_{ij} = \int_{\uS^2} \partial_\phi G_i \partial_\phi G_j \,
	\dd \mu^\pm.
\]
When $p_i$ and $p_j$ are antipodal (e.g. $(i,j) = (1,7)$) or on a horizontal
face diagonal (e.g. $(i,j) = (1,3)$), one easily sees that the integrand
$\partial_\phi G_i \partial_\phi G_j \le 0$ everywhere on the sphere, so
$H^\pm_{ij}$ is negative.  When $p_i$ and $p_j$ are vertical neighbors (e.g.
$(i,j) = (1,5)$), one easily sees that the integrand $\partial_\phi G_i
\partial_\phi G_j \ge 0$ everywhere on the sphere, so $H^\pm_{ij}$ is positive.
Otherwise, the following lemma would be useful:

\begin{lemma}
	Let $f(\theta, \phi)$ be a function defined on $\uS^2$ that is odd in $\phi$,
	and $\rho$ be a positive function on $\uS^2$ with
	\[
		\rho(\theta, \phi) = \rho(\theta, \phi+\tfrac{\pi}{2}) = \rho(\pi - \theta, \phi).
	\]
	Assume that, for every $0 < \theta < \pi$ and $0 < \phi < \pi/2$, we have
	\[
		f(\theta, \phi) > f(\theta, \pi - \phi), \qquad
		f(\theta, \pi/2 - \phi) > f(\theta, \pi/2 + \phi).
	\]
	Then the integrals
	\begin{align*}
		K_1 =& \int_{\uS^2} f(\theta, \phi) f(\theta, \phi - \tfrac{\pi}{2}) \rho \, \dd A < 0,\\
		K_2 =& \int_{\uS^2} f(\theta, \phi) f(\pi - \theta, \phi - \tfrac{\pi}{2}) \rho \, \dd A < 0.
	\end{align*}
\end{lemma}

\begin{proof}
	We decompose the sphere into four quadrants.  Then
	\begin{align*}
		K_1 = \int_0^\pi \sin\theta \, d\theta \int_{0}^{\tfrac{\pi}{2}} & \rho \, d\phi
		\big(f(\theta, \phi) f(\theta, \phi-\tfrac{\pi}{2}) + f(\theta, \phi+\tfrac{\pi}{2}) f(\theta, \phi)\\
		&+ f(\theta, \phi+\pi) f(\theta, \phi+\tfrac{\pi}{2}) + f(\theta, \phi-\tfrac{\pi}{2}) f(\theta, \phi+\pi) \big)\\
		= \int_0^\pi \sin\theta \, d\theta \int_{0}^{\tfrac{\pi}{2}} \rho \, d\phi
		& \big( f(\theta, \phi) + f(\theta, \phi + \pi) \big) \big( f(\theta, \phi - \tfrac{\pi}{2}) + f(\theta, \phi + \tfrac{\pi}{2}) \big)\\
		= \int_0^\pi \sin\theta \, d\theta \int_{0}^{\tfrac{\pi}{2}} \rho \, d\phi
		& \big( f(\theta, \phi) - f(\theta, \pi - \phi) \big) \big(-f(\theta, \tfrac{\pi}{2} - \phi) + f(\theta, \tfrac{\pi}{2} + \phi) \big)
	\end{align*}
	For $0 < \phi < \tfrac{\pi}{2}$, by the assumed inequalities, $K_1$ is
	negative.  On the other hand,
	\begin{align*}
		K_2 = \int_0^{\tfrac{\pi}{2}} & \sin\theta \, d\theta \int_{0}^{\tfrac{\pi}{2}} \rho \, d\phi \big(
			f(\theta, \phi) f(\pi - \theta, \phi-\tfrac{\pi}{2}) + f(\pi - \theta, \phi) f(\theta, \phi-\tfrac{\pi}{2})\\
			&+f(\theta, \phi+\tfrac{\pi}{2}) f(\pi - \theta, \phi) + f(\pi - \theta, \phi+\tfrac{\pi}{2}) f(\theta, \phi)\\
			&+f(\theta, \phi+\pi) f(\pi - \theta, \phi+\tfrac{\pi}{2}) + f(\pi - \theta, \phi+\pi) f(\theta, \phi+\tfrac{\pi}{2})\\
			&+f(\theta, \phi-\tfrac{\pi}{2}) f(\pi - \theta, \phi+\pi) + f(\pi - \theta, \phi-\tfrac{\pi}{2}) f(\theta, \phi+\pi)
		\big)\\
		= \int_0^{\tfrac{\pi}{2}} & \sin\theta \, d\theta \int_{0}^{\tfrac{\pi}{2}} \rho \, d\phi \\ & \Big[
			\big(f(\pi - \theta, \phi) + f(\pi - \theta, \phi + \pi) \big) \big( f(\theta, \phi - \tfrac{\pi}{2}) + f(\theta, \phi + \tfrac{\pi}{2}) \big) \\
			&+\big(f(\theta, \phi) + f(\theta, \phi + \pi) \big) \big( f(\pi - \theta, \phi - \tfrac{\pi}{2}) + f(\pi - \theta, \phi + \tfrac{\pi}{2}) \big) 
		\Big] \\
		= \int_0^{\tfrac{\pi}{2}} & \sin\theta \, d\theta \int_{0}^{\tfrac{\pi}{2}}\rho \, d\phi \\ & \Big[
			\big(f(\pi - \theta, \phi) - f(\pi - \theta, \pi - \phi) \big) \big( -f(\theta, \tfrac{\pi}{2} - \phi) + f(\theta, \phi + \tfrac{\pi}{2}) \big) \\
			&+\big(f(\theta, \phi) - f(\theta, \pi - \phi) \big) \big( -f(\pi - \theta, \tfrac{\pi}{2} - \phi) + f(\pi - \theta, \phi + \tfrac{\pi}{2}) \big) 
		\Big]
	\end{align*}
	and the integrand is negative by the assumed inequalities.
\end{proof}

Using this lemma, one easily proves that if $p_i$ and $p_j$ are horizontal
neighbors (e.g. $(i,j) = (1,2)$) or on a vertical face diagonal (e.g. $(i,j) =
(1,6)$), we have $H^\pm_{ij} < 0$.

For $T_{1u}$, the integrand does not involve any pair of vertical neighbors, so
we can already conclude that the Hessian is positive definite on the irreps
$T_{1u}$.

\subsection{The definiteness on \texorpdfstring{$T_{2u}$}{T2u}}

For $T_{2u}$ we place the cube so that a pair of antipodal vertices are at the
north and the south pole.  So
\begin{align*}
	p_1 &= (0, 0),&
	p_5 &= (\pi, 0),\\
	p_2 &= (\theta_0, 0),&
	p_6 &= (\pi - \theta_0, \pi/3),\\
	p_3 &= (\theta_0, 2\pi/3),&
	p_7 &= (\pi - \theta_0, \pi),\\
	p_4 &= (\theta_0, 4\pi/3),&
	p_8 &= (\pi - \theta_0, 5\pi/3).
\end{align*}
where $\theta_0 = \arccos(-1/3)$.  Then a generator for $T_{2u}$ is given by
\[
	v_i = \frac{\partial_\phi}{\sin\theta}\Big|_{p_i} \text{ for } i \in \{2, 3, 4\},\qquad
	v_i =-\frac{\partial_\phi}{\sin\theta}\Big|_{p_i} \text{ for } i \in \{6, 7, 8\},\qquad
	v_1 = v_5 = 0.
\]
Note that the vertices $p_1$ and $p_5$ (at the poles) are actually fixed under
this generator.  The Hessian becomes
\[
	\Hess_{\! p} \log I^\pm (v, v)
	= -16\pi^2 \int_{\uS^2} \partial_\phi G_A \partial_\phi G_B \, \dd \mu^\pm,
\]
where $A=\{2,3,4\}$ and $B=\{6,7,8\}$.  One easily notices that the integrand
$\partial_\phi G_A \partial_\phi G_B \le 0$ everywhere on the sphere.  So the
Hessian is positive definite on $T_{2u}$.

\subsection{The definiteness on \texorpdfstring{$T_{2g}$}{T2g}}
\label{app:T2g}

This part contains the most involved computation.  The argument was found with the assistance of ChatGPT 5.6 Sol.

\subsubsection{Preparation}

Given a point $p \in \mathbb{S}^2$ with polar angle $\theta_0$ and azimuth
angle $\phi_0$, the Green function $G(\theta,\phi)$ centered at $p$ can be
written as (compare~\eqref{eq:Green})
\[
	\frac{1}{4\pi} \log \big(1 - \cos\theta \cos\theta_0 - \sin\theta \sin\theta_0 \cos(\phi-\phi_0)\big).
\]
After grouping the eight points into four antipodal pairs, we obtain
\begin{align*}
	\widehat G_1 = G_1+G_7 &= \frac{1}{4\pi}\log\frac{3-\zeta_1^2}{3},& \zeta_1 = \sqrt{2} x + z,\\
	\widehat G_2 = G_2+G_8 &= \frac{1}{4\pi}\log\frac{3-\zeta_2^2}{3},& \zeta_2 = \sqrt{2} y + z,\\
	\widehat G_3 = G_3+G_5 &= \frac{1}{4\pi}\log\frac{3-\zeta_3^2}{3},& \zeta_3 = \sqrt{2} x - z,\\
	\widehat G_4 = G_4+G_6 &= \frac{1}{4\pi}\log\frac{3-\zeta_4^2}{3},& \zeta_4 = \sqrt{2} y - z,
\end{align*}
where
\[
	x = \sin\theta\cos\phi, \qquad
	y = \sin\theta\sin\phi, \qquad
	z = \cos\theta.
\]
Then we have
\[
	\exp(u) = \exp\Big(2\pi\sum \widehat G_i\Big)
	= \frac{\sqrt{(3-\zeta_1^2)(3-\zeta_2^2)(3-\zeta_3^2)(3-\zeta_4^2)}}{9},
\]
and
\begin{align*}
	\frac{\partial_\phi}{\sin\theta} \widehat G_1 &= \frac{\zeta_1 \sin\phi}{\sqrt{2}\pi(3 - \zeta_1^2)}, &
	\frac{\partial_\phi}{\sin\theta} \widehat G_2 &= \frac{-\zeta_2 \cos\phi}{\sqrt{2}\pi(3 - \zeta_2^2)}, \\
	\frac{\partial_\phi}{\sin\theta} \widehat G_3 &= \frac{\zeta_3 \sin\phi}{\sqrt{2}\pi(3 - \zeta_3^2)}, &
	\frac{\partial_\phi}{\sin\theta} \widehat G_4 &= \frac{-\zeta_4 \cos\phi}{\sqrt{2}\pi(3 - \zeta_4^2)}.
\end{align*}
So we need to prove that
\[
	H^+ = \int_0^\pi \dd\theta \int_0^{2\pi}
	\frac{\dd\phi}{18\pi^2} \frac{P_1}{\sqrt{1-z^2} Q^{1/2}},\qquad
	H^- = \int_0^\pi \dd\theta \int_0^{2\pi}
	\frac{9 \dd\phi}{2\pi^2} \frac{P_1}{\sqrt{1-z^2} Q^{3/2}}
\]
are both negative, where
\begin{align*}
	P_1 &= (\zeta_1y(3-\zeta_2^2) - \zeta_2x(3-\zeta_1^2))(\zeta_3y(3-\zeta_4^2)-\zeta_4x(3-\zeta_3^2)),\\
	Q &= (3-\zeta_1^2)(3-\zeta_2^2)(3-\zeta_3^2)(3-\zeta_4^2).
\end{align*}

\subsubsection{Symmetrization}

The integrals are invariant under the symmetries
\begin{itemize}
	\item $x \leftrightarrow -x$ that induces $\zeta_1 \leftrightarrow -\zeta_3$,
	\item $y \leftrightarrow -y$ that induces $\zeta_2 \leftrightarrow -\zeta_4$,
	\item $z \leftrightarrow -z$ that induces $(\zeta_1, \zeta_2) \leftrightarrow (\zeta_3, \zeta_4)$, and
	\item $x \leftrightarrow y$ that induces $(\zeta_1, \zeta_3) \leftrightarrow (\zeta_2, \zeta_4)$.
\end{itemize}
After symmetrization, we have
\begin{align*}
	H^+ &= 8 \int_0^{\pi/2} \dd\theta \int_0^{\pi/4}
	\frac{\dd\phi}{18\pi^2} \frac{P}{\sqrt{1-z^2} Q^{1/2}},\\
	H^- &= 8 \int_0^{\pi/2} \dd\theta \int_0^{\pi/4}
	\frac{9 \dd\phi}{2\pi^2} \frac{P}{\sqrt{1-z^2} Q^{3/2}},
\end{align*}
where $P=P_1 + P_2$ and
\[
	P_2 = (\zeta_3y(3-\zeta_2^2) - \zeta_2x(3-\zeta_3^2))(\zeta_1y(3-\zeta_4^2)-\zeta_4x(3-\zeta_1^2)).
\]
Define $s = x^2 + y^2 = 1-z^2$ and $t = x^2 - y^2 = s \cos(2\phi)$.  Then $P$
and $Q$ can be written as polynomials of $s$ and $t$ as follows.
\begin{align*}
	P &= -4t^4 + 2 (7 s^2 - 25 s + 20) t^2 + 8 (s^2+s)(s^2-1)\\
	% (2s+t-1)(s-t)((2+t)^2-4(s-t)(1-s)) + (2s-t-1)(s+t)((2-t)^2-4(s+t)(1-s))
	% - 2 (s^2 - t^2) (-4 + 2 s - t) (-4 + 2 s + t)\\
	Q &= ((2-t)^2-4(s+t)(1-s))((2+t)^2-4(s-t)(1-s)),
\end{align*}
This is first suggested by ChatGPT 5.6 Sol, then verified by Mathematica and
manually by the author.

\subsubsection{Bounding $P$ and $Q$}

We point out the following bounds of $P$ and $Q$:

\begin{description}
	\item[$P > 0$ only when $s > \tfrac{8}{9}$]
		Write
		\[
			P = -4\tau^2 + 2 \tau (7 s^2 - 25 s + 20) + 8 s (s + 1) (s^2 - 1)
		\]
		with $\tau = t^2$.  We compute the partial derivative
		\[
			\frac{\partial P}{\partial \tau} = -8\tau + 2 (7 s^2 - 25 s + 20).
		\]
		For $0 \le t \le s \le 8/9$, we have
		\[
			\frac{\partial P}{\partial \tau} \ge 2 (3s^2 - 25 s + 20) > 0,
		\]
		so
		\[
			P(t,s) \le P(s,s) = 2 s (2-3s)^2 (s-1) \le 0.
		\]
		This proves that $P \le 0$ when $s \le 8/9$.

	\item[$P \le s$ when $\tfrac{8}{9} \le s \le 1$]
		\[
			\frac{P}{s} = -4\frac{\tau^2}{s} + 2 \frac{\tau}{s} (7 s^2 - 25 s + 20) + 8 (s + 1) (s^2 - 1)
		\]
		has an unrestricted maximum
		\[
			8 (s+1) (s^2-1) + \frac{(7 s^2 - 25 s + 20)^2}{4s},
		\]
		which is indeed $\le 1$ for $8/9 \le s \le 1$.

	\item[$Q \le 16$]  Write
		\[
			Q = \tau^2 - 8\tau (s^2-7s+7) + 16(s^2-s+1)^2
		\]
		with $\tau = t^2$.  Then the partial derivative
		\begin{equation}\label{eq:dQdt}
			\frac{\partial Q}{\partial \tau} = 2\tau - 8(s^2-7s+7)
		\end{equation}
		is negative for $0 \le t \le s \le 1$.  So we have
		\[
			Q(t,s) \le Q(0,s) = 16 (s^2-s+1)^2 \le 16,
		\]
		with equality when $t = 0$ and $s=0$ or $1$.

	\item[$Q \ge (\tfrac{52}{27})^2$ when $s \ge \tfrac{8}{9}$] As the partial
		derivative~\eqref{eq:dQdt} is negative, we have for $0 \le t \le s$ and
		$8/9 \le s \le 1$, we have
		\[
			Q(t,s) \ge Q(s,s) = (-4+4s+3s^2)^2 \ge \Big(\frac{52}{27}\Big)^2,
		\]
		with equality when $t = s = 8/9$.
\end{description}

\subsubsection{Bounding the integrals}

Finally, we estimate the integrals
\[
	H^+ = 8 \int_0^1 \dd z \int_0^{\pi/4}
	\frac{\dd\phi}{18\pi^2} \frac{P}{s Q^{1/2}},\qquad
	H^- = 8 \int_0^1 \dd z \int_0^{\pi/4}
	\frac{9 \dd\phi}{2\pi^2} \frac{P}{s Q^{3/2}}.
\]
Write $P^+ = \max(P, 0)$.  Then
\[
	\frac{P}{\sqrt{Q}^\alpha} \le \frac{P}{4^\alpha} + \Big(\big(\frac{27}{52}\big)^\alpha - \frac{1}{4^\alpha}\Big) P^+.
\]
where $\alpha = 1$ for $H^+$ and $\alpha = 3$ for $H^-$.
Therefore,
\begin{align}
	\int_0^1 \dd z \int_0^{\pi/4} 
	\dd\phi \frac{P}{s \sqrt{Q}^\alpha} \le &
	\frac{1}{4^\alpha} \int_0^1 \dd z \int_0^{\pi/4} \dd\phi \frac{P}{s}\label{eq:fstint}\\
	& + \Big(\big(\frac{27}{52}\big)^\alpha - \frac{1}{4^\alpha}\Big) \int_0^1 \dd z \int_0^{\pi/4} \dd\phi \frac{P^+}{s}\label{eq:sndint}
\end{align}

The integral on the right-hand-side of~\eqref{eq:fstint} can be explicitly
computed
\[
	\int_0^1 \dd z \int_0^{\pi/4} \dd\phi \frac{P}{s} = -\frac{76}{105}\pi.
\]
The integral~\eqref{eq:sndint} can be restricted to the region $8/9 \le s \le
1$, where we know that $P/s \le 1$, so
\[
	\int_0^1 \dd z \int_0^{\pi/4} \dd\phi \frac{P^+}{s} = 
	\int_0^{1/3} \dd z \int_0^{\pi/4} \dd\phi \frac{P^+}{s} \le \frac{\pi}{12}.
\]
Putting these all together yields
\[
	\int_0^1 \dd z \int_0^{\pi/4} 
	\dd\phi \frac{P}{s \sqrt{Q}^\alpha} \le
	- \frac{1}{4^\alpha}\frac{76}{105}\pi + \Big(\big(\frac{27}{52}\big)^\alpha - \frac{1}{4^\alpha}\Big) \frac{\pi}{12}.
\]
This is negative for $\alpha=1$ and $\alpha=3$, thereby proving that $H^\pm$
are both negative.

\end{document}